\documentclass[12pt,reqno]{amsart}
\usepackage{amsmath,amsthm,amsfonts,amssymb}

\newtheorem{theorem}{Theorem}[section]

\newtheorem{lemma}[theorem]{Lemma}
\newtheorem{cor}[theorem]{Corollary}

\newtheorem{prop}[theorem]{Proposition}
\theoremstyle{remark}
\newtheorem{remark}[theorem]{Remark}

\newcommand{\Z}{\mathbb{Z}}
\newcommand{\R}{\mathbb{R}}

\newcommand*{\be}{\begin{equation}}
\newcommand*{\ee}{\end{equation}}
\newcommand{\nn}{\nonumber}
\providecommand{\abs}[1]{\lvert#1\rvert}
\newcommand*{\fl}[1]{\lfloor{#1}\rfloor}
\def\ind{\mathbf{1}}
\DeclareMathOperator{\Var}{Var}
\DeclareMathOperator{\Cov}{Cov}

\newcommand{\ep}{\varepsilon}
\newcommand\eps{\varepsilon}
\renewcommand\epsilon{\varepsilon}
\newcommand*{\Ev}{\mathbf E}
\newcommand*{\Vv}{{\text{\bf Var}}} 

\newcommand*{\Pv}{\mathbf P}

\def\sc{\mathbf x}
\def\N{\mathbb N}

 \def\wt{\widetilde}   

\begin{document}

\title[Scaling of KPZ/Stochastic Burgers]{Scaling exponent 
for the Hopf-Cole solution 
of KPZ/Stochastic Burgers}

\author[M. Bal\'azs]{M. Bal\'azs}
\address{Department of Stochastics\\
Budapest University of Technology and Economics}
\thanks{M. Bal\'azs is  supported by   the Hungarian Scientific
Research Fund (OTKA) grants K-60708, F-67729, by the Bolyai Scholarship
of the Hungarian Academy of Sciences, and by Morgan Stanley Mathematical
Modeling Center.}
\email{balazs@math.bme.hu}
\urladdr{www.math.bme.hu/$\sim$balazs}
\author[
J. Quastel]{
J. Quastel}
\address{Departments of Mathematics and Statistics\\
University of Toronto
}
\thanks{J. Quastel is supported by the Natural Sciences and Engineering Research Council of Canada. }
\email{quastel@math.toronto.edu}
\urladdr{www.math.toronto.edu/quastel}
\author[T. Sepp\"al\"ainen]{T. Sepp\"al\"ainen}
\address{Department of Mathematics\\
University of Wisconsin--Madison }
\thanks{T. Sepp\"al\"ainen is supported by National Science Foundation grant DMS-0701091 and by the Wisconsin Alumni Research Foundation.} 
\email{seppalai@math.wisc.edu}
\urladdr{www.math.wisc.edu/$\sim$seppalai}


\date{\today}




\begin{abstract}
We consider the stochastic heat equation
\[\label{1}\partial_tZ= \partial_x^2 Z - Z \dot W
\]
on the real line, where $\dot W$ is  space-time white noise.  $h(t,x)=-\log Z(t,x)$ is interpreted as a solution of the KPZ
equation, and $u(t,x)=\partial_x h(t,x)$ as a solution of the stochastic Burgers equation. We take  $Z(0,x)=\exp\{B(x)\}$ where $B(x)$ is a two-sided Brownian motion, corresponding to the stationary solution of the stochastic Burgers equation. We show that there
exist $0< c_1\le c_2 <\infty$ such that
\[
c_1t^{2/3}\le \Var (\log Z(t,x) )\le c_2 t^{2/3}.
\]
Analogous results are obtained for some moments of the correlation functions of $u(t,x)$. In particular, it is
shown that the excess diffusivity satisfies
\[
c_1t^{1/3}\le D(t) \le c_2 t^{1/3}.\] 
The proof uses approximation  by weakly asymmetric simple exclusion processes, for which we obtain the
microscopic analogies of the results  by coupling.
\end{abstract}
\maketitle

\section{Introduction}
\label{intro}
\setcounter{equation}{0}

\subsection{Physical background.}  The Kardar-Parisi-Zhang (KPZ) equation
\cite{KPZ}
 is a formal stochastic partial differential equation for
a random function $h(t,x)$, $t>0$, $x\in {\mathbb R}$,
\begin{equation}
\partial_t h = -\lambda (\partial_xh)^2 + \nu \partial_x^2 h + \sigma \dot W
\label{KPZ0}
\end{equation}
where $\nu>0$ and $\sigma,\lambda\neq 0$ are fixed parameters and $\dot W(t,x)$ is  Gaussian space-time white noise
\[
E[ \dot W(t,x) \dot W(s,y)] = \delta(t-s)\delta(y-x).
\]
 It is widely studied in physics as a model of
randomly growing interfaces. 
  The derivative $u=\partial_x h$ should satisfy the  stochastic Burgers equation,
\begin{equation}\label{SBE}
\partial_t u = - \lambda\partial_x u^2 + \nu \partial_x^2 u + \sigma \partial_x \dot W.
\end{equation}

Using renormalization group  methods physicists have computed the dynamic scaling exponent (\cite{FNS}, \cite{KPZ}, \cite{BSt})
\[
z=3/2.
\]
 Roughly, this means that one expects non-trivial behavior under the rescaling
\[
h_\eps(t,x)  =\eps^{1/2} h( \eps^{-3/2} t, \eps^{-1}x).
\] 
For the totally asymmetric exclusion process and the polynuclear growth models, which can
 be thought of as discretizations of (\ref{KPZ0}),  it is now known  rigorously \cite{FS, PS} that in a weak sense,
 \begin{equation}\label{sc}
 \Var ( h_\eps(t,x))   \simeq  t^{2/3} g_{\rm sc}(t^{-2/3} (x- {\rm v} t))
 \end{equation}
 for an explicit ${\rm v}$ and scaling function $g_{\rm sc}$ related to the Tracy-Widom distribution.
 Note that these models are in some sense exactly solvable. 

(\ref{KPZ0}) and (\ref{SBE}) are ill-posed because the quadratic non-linear term cannot possibly make sense for 
a typical realization, which, in the case of (\ref{KPZ0}) is expected to look, in $x$, locally, like a Brownian motion with variance $\nu^{-1}\sigma^2$.  Formally  applying the
Hopf-Cole transformation \begin{equation}\label{hc0}
Z(t,x)= \exp\{-\lambda\nu^{-1} h (t,x)\}
\end{equation}
to (\ref{KPZ0}) leads to the stochastic heat equation\begin{equation}\label{heat}
\partial_tZ = \nu \partial_x^2 Z  - \lambda\nu^{-1}\sigma  Z\dot W.
\end{equation}
 The advantage is that (\ref{heat}) is well-posed \cite{W}.  We do not attempt to justify
 the manipulations leading to (\ref{heat}).  We define $h$ and $u=\partial_x h$ through (\ref{hc0}).
 These {\it Hopf-Cole solutions} are expected to be the physically relevant solutions of (\ref{KPZ0}) and (\ref{SBE}).  

$Z$ can  be thought of as an asymptotic model of a directed polymer.
  There is a Feynman-Kac formula   
\begin{equation}\label{zed0}
Z(t,x)= E^\nu_x[ :\!\exp\{-\beta \int_0^t \dot W( s, b(s)) ds \}\!\!: Z(0,b(t))]
\end{equation}
where the expectation is over an independent Brownian motion $b(s)$, $s\ge 0$ starting at $x$,
of variance $\nu$, $\beta = \lambda\nu^{-1}\sigma$, and $:\!\!\exp \!\!:$ is the Wick-ordered exponential (see  \cite{NR} for details).   
The reason  we write  (\ref{zed0}) is to draw attention to the analogy with directed polymers, a typical model being
\begin{eqnarray}\label{dp}
&z(n,x)= E_{x}[ \exp\{-\beta \sum_{m=1}^n X(m, s_m) \}]&
\end{eqnarray} 
where $X(m,r)$, $m\in\{1,2,\ldots\}$, $r\in\{ \ldots,-1,0,1,\ldots\}$ are independent and identically distributed random variables, and $s_m$ is a simple random walk starting at $x$.  Assuming reasonable decay on the tails of the $X$'s, it is expected \cite{KS} that for any $\beta$,
\begin{equation}\label{dis}
\Var  (\log z(n,0)) \sim cn^{2\chi }\quad\text{with}\quad \chi =1/3.
\end{equation}
Little is known rigorously.  Bounds analogous to $\chi\in [ 3/10, 1/2]$ are obtained in 
 \cite{BTV},  \cite{LNP}, \cite{M}, \cite{P}, and  \cite{W}.
The closest results with $\chi=1/3$ are those of \cite{BDJ, J} for certain last passage percolation models, which are obtained in the $\beta\to \infty$ (zero-temperature) limit.   Note the  
contrast with dimensions $d\ge 3$ where the polymer is known to
be diffusive for  small $\beta$.

\subsection{Mathematical background}
We now survey what is known rigorously about (\ref{KPZ0}).  
In terms of well-posedness, the technology at the present time \cite{DePZ} can only
handle far smoother noise terms than the white noise.  An unusual type of Wick product
version of the problem has been introduced \cite{HOUZ}.
 But besides requiring fairly smooth noises, this does
{\em not} have the scaling expected \cite{Ch}, and is therefore believed {\em not} to be
physically relevant.

The idea in \cite{BG}, which leads to what appears to be the physically relevant solution, is to smooth out the white noise in space a little, and then
use the Hopf-Cole transformation and the tractability of (\ref{heat})
to remove the cutoff.  As this is done, one finds one has to subtract a large constant from the equation.  The resulting {\em Hopf-Cole  solution of} (\ref{KPZ0})
is given  explicitly as the logarithm of the well-defined solution of (\ref{heat}).  We now recall the details.

Let $W(t)$, $t \ge 0$, be the cylindrical Wiener process, i.e. the continuous Gaussian process taking values in  $H^{-1/2-}_{\rm loc} (\mathbb R)=\cap_{\alpha<-1/2} H^{\alpha}_{\rm loc}(\mathbb R)$  with
\begin{equation}
E[ \langle \varphi ,W(t)\rangle \langle \psi, W(s)\rangle ] =\min(t,s) \langle \varphi, \psi\rangle 
\end{equation}
for any $\varphi,\psi\in C_c^\infty(\mathbb R)$, the smooth functions with compact support in $\mathbb R$.
Here $H^{\alpha}_{\rm loc}(\mathbb R)$, $\alpha<0$,  consists of
distributions $f$ such that for any $\varphi\in C_c^\infty(\mathbb R)$,
$\varphi f$ is in the standard Sobolev space $H^{-\alpha}(\mathbb R)$, i.e. the dual of $H^{\alpha}(\mathbb R)$ under the $L^2$ pairing.
$H^{-\alpha}(\mathbb R)$ is the closure of $C_c^\infty(\mathbb R)$ under 
the norm $\int (1+ |t|^{-2\alpha}) |\hat f(t)|^2dt$ where $\hat f$ denotes 
the Fourier transform.

The distributional time derivative $\dot W(t,x)$ 
 is space-time white noise.
 These Sobolev spaces are
the natural home, as can be seen
 by restricting to a finite box and writing a Fourier series for 
 white noise,  with independent and identically distributed
 Gaussian variables  as Fourier coefficients.   Note the mild abuse of notation for the sake of clarity, as we write $\dot W(t,x)$ even though it is a distribution on $(t,x)\in [0,\infty)\times \mathbb R$ as opposed to  a classical function of $t$ and $x$. 

  Let ${\mathcal F}(t)$, $t\ge 0$, be the natural filtration, i.e. the smallest $\sigma$-field
 with respect to which  $W(s)$ are measurable for all $0\le s\le t$.  Let ${\mathcal G}\in C_c^\infty({\mathbb R})$ be an even, non-negative function with total integral $1$   and for $\kappa>0$, 
${\mathcal G}_\kappa(x) = \kappa^{-1} {\mathcal G}( \kappa^{-1} x )$.  The mollified Wiener process is given by $W^\kappa (t,x) = \langle\tau_{-x}{\mathcal G}_\kappa, W(t)\rangle $ where $\tau_{-x} f(y) = f(y+x)$ is the shift operator.  The distributional time derivative of $W^\kappa (t,x)$ is the Gaussian field with covariance
\begin{equation}C_\kappa(x-y) \delta(t-s)= E[ \dot W^\kappa(t,x) \dot W^\kappa(s,y) ]
\end{equation} 
where
\begin{eqnarray}
&C_\kappa(x) = \kappa^{-1} \int {\mathcal G}(\frac{\kappa^{-1}}{2} x -y ){\mathcal G}(-\frac{\kappa^{-1}}{2} x -y )dy, \quad C_\kappa(0)=\kappa^{-1}\|\mathcal G\|^2_2.&\nonumber
\end{eqnarray}

Assume the initial data $h(0,x)$ is a random continuous function on $\mathbb R$ such that for each $p>0$ there is an
$a=a(p)<\infty$ such that 
\begin{eqnarray}&\sup_{
x\in {\mathbb R}} e^{-a |x|} E[ e^{-p\,h(x)}]  <\infty
.&
\label{idat}
\end{eqnarray}
The mollified KPZ equation is
\begin{equation}
\partial_t h^\kappa = -\lambda[(\partial_xh^\kappa)^2 -C_\kappa(0)] + \nu\partial_x^2 h^\kappa + \sigma \dot W^\kappa.
\label{KPZ}
\end{equation}

 The following  summarizes previous results, mostly from \cite{BG}.

\begin{prop} \label{thm1} Let $h(0,x)$ be a random initial continuous function satisfying (\ref{idat}) and independent of the white noise $\dot W$. 

\noindent{\it 1.}  For each $\kappa>0$, there exists a unique continuous Markov process $h_\kappa(t)$, with
probability distribution $P_\kappa$ on  $C([0,T], C(\mathbb R))$, adapted to  $\tilde{\mathcal F}(t)= \sigma( h_0, W_s, s\le t)$, $t\ge 0$, and solving (\ref{KPZ}).  

\noindent{\it 2. }  The process $Z^\kappa(t,x)= \exp\{-\lambda\nu^{-1} h_\kappa(t,x)\}$ is the unique adapted (mild) solution of the It\^o  equation,
\begin{equation}
\partial_tZ^\kappa = \nu \partial_x^2Z^\kappa  + \lambda\nu^{-1} \sigma Z^\kappa \dot W^\kappa, \qquad Z^\kappa(0,x) = \exp\{-\lambda\nu^{-1} h(0,x)\}
\end{equation}
and has the representation \begin{equation}\nonumber
Z^\kappa(t,x)= E_x[ \exp\{-\beta\int_0^t \dot W^\kappa( s, b(s)) ds - \frac12 C_\kappa(0) t \} Z^\kappa(0,b(t))].
\end{equation}
{\it 3.} $Z^\kappa(t,x)\to Z(t,x)$  almost surely, uniformly on compact sets of $[0,\infty)\times \mathbb R$ where $Z$ is the unique adapted (mild) solution of (\ref{heat}) with $ Z(0,x) = \exp\{-\lambda\nu^{-1} h(0,x)\}$.
Furthermore,
   $Z(t,x)>0$ for all $t\ge 0$, $x\in\mathbb R$ almost surely \cite{Mu}. 
   
   \noindent
{\it 4.} The $P_\kappa$ are tight as measures on  $C([0,T], C(\mathbb R))$.  The limit
process $h(t)$ coincides with 
\begin{equation}\label{def}
h(t,x) = -\lambda^{-1}\nu \log Z(t,x).
\end{equation}

\noindent{\it 5.}    Define
  \begin{equation}
  u = \partial_x h
  \end{equation}
  in the sense of distributions.  If we start $Z$ with initial data $Z(0,x) = \exp\{B(x)\}$ where $B(x)$ is a two sided Brownian motion independent of $W$ with variance $\nu^{-1}\sigma^2$, then $u$ is stationary in both space and time.   In this sense, Gaussian white
  noise with variance $\nu^{-1}\sigma^2$ is invariant for (\ref{SBE}).  Stationarity here means stationarity of
  $\langle \tau_x\varphi , u(t)\rangle$ for smooth functions of compact support $\varphi(x)$  where
  $(\tau_x\varphi)(y)= \varphi(y-x)$.  The corresponding $h(t,x)$ and $Z(t,x)$ are not stationary in time with these initial
  data, but the increments $D_\delta h(t,x)= h(t,x+\delta)-h(t,x)$  are space and time stationary.  
  
  \noindent
{\it 6.} Let  $h$ be the Hopf-Cole solution of (\ref{KPZ0}), as in (\ref{def}).  Then \begin{equation}
h^\gamma(t,x) = \gamma^{-\alpha} h( \gamma^{-\beta} t, \gamma^{-1} x)
\end{equation} is the Hopf-Cole solution of (\ref{KPZ0}) with
new coefficients \begin{equation}\label{rescale}\lambda_\gamma  =  \gamma^{\alpha-\beta+2}\lambda,\quad
\nu_\gamma = \gamma^{-\beta+2}\nu, \quad \sigma_\gamma =
\gamma^{\frac{-\beta+1-2\alpha}{2}}\sigma.
\end{equation}
If we start (\ref{SBE}) in equilibrium, ie. $u(0)$ is a white noise
with variance $\nu^{-1}\sigma^2$, then the time reversed process
$u(T-t)$, $t\in[0,T)$ is a solution of (\ref{SBE}) with $\lambda$ replaced
by $-\lambda$, and the spatially reversed process $h(t,-x)\stackrel{\rm dist}{=}h(t,x)$.
\end{prop}

Note that \cite{BG} only consider the case $\lambda=\nu=1/2$ and $\sigma=1$, but their proofs work in general.  {\it 6} is not stated there, but it follows readily from their methods. 

The {\em Hopf-Cole solution}  (\ref{def}) {\em of KPZ} (\ref{KPZ0}) is obtained as a limit of solutions of (\ref{KPZ}), i.e. after subtraction
of a divergent term $C_\kappa(0)$.  An important open problem is to show that 
a corresponding version of (\ref{KPZ0}) with an appropriate Wick
ordered nonlinearity is well-posed.  We do not address this issue here.  Since (\ref{def}) is expected to be the relevant solution, we simply study
it directly.

\subsection{Statement of results.}  We can now state our main results about $h$.

\begin{theorem}  Let  $h(t,x)$ be the Hopf-Cole solution of (\ref{KPZ0}) as in (\ref{def}) with $Z(t,x)$ the solution of (\ref{heat}), with  initial data $Z(0,x) = \exp\{B(x)\}$ where $B(x)$ is a two sided Brownian motion independent of $W$ with variance $\nu^{-1}\sigma^2$.  Let  $\Var (h(t,x))$ denote the variance of $h(t,x)$.  $\Var (h(t,x))$ is a symmetric 
function of $x$, non-decreasing in $|x|$, and
\begin{equation}\label{pospos}
\Var (h(t,x)) - |x|\ge 0.
\end{equation}
Furthermore, there exist $ c_0=c_0(\sigma,\nu,\lambda)<\infty$, and $C_1=C_1(\sigma,\nu,\lambda)<\infty$,
$C_2=C_2(m,\sigma,\nu,\lambda)<\infty$,
%
%
such that 
for $
t\ge c_0
$
we have \begin{equation}\label{bdsonh}
C_1 t^{2/3} \le 
 \Var (h(t,0))  \le C_2 t^{2/3}
 \end{equation}
 and, for $1<m<3$, 
  \begin{equation}\label{bdsonh1}
C_1 t^{2m/3} \le
\int |x|^{m-2}  \left[ \Var (h(t,x)) - |x| \right] dx  \le C_2 t^{2m/3}.
 \end{equation}\label{theoremh}
\end{theorem}

\begin{remark}
 The dependence of the constants $c_0$ and $C$ on  $m,\sigma,\nu,\lambda$ is as follows:  We can take
 \begin{equation}\nonumber
 c_0= \tilde c_0\sigma^{-4}\nu^5\lambda^{-4},~~ C_1=\tilde C_1 \sigma^{\frac{2m}{3} + 2} \nu^{-\frac{m}{3} -1}\lambda^{\frac{2m}{3}}, ~~ C_2=  C(m)\sigma^{\frac{2m}{3} + 2} \nu^{-\frac{m}{3} -1}\lambda^{\frac{2m}{3}} 
 \end{equation} for some $\tilde c_0, \tilde C_1\in (0,\infty)$ and $C(m)= \frac{C}{m(3-m)(m-1)}$, $1<m<3$ and $C(1)=C/4$.  Here $m=1$ refers to (\ref{bdsonh}).  The lower bound holds for all $m\ge 1$.
\end{remark}
 
 \begin{remark}
Theorem \ref{theoremh} tells us that  the
 scaling exponent for the quantities studied
 here follows the physical prediction, as can be seen by integrating
 (\ref{sc}).  This 
 provides considerable support
for the notion that this process $h$ 
 is the sought after solution of (\ref{KPZ0}), even though 
 is not  known presently how to show directly that $h(t,x)=-\lambda^{-1}\nu\log Z(t,x)$ solves (\ref{KPZ0}), or indeed, what it means to solve (\ref{KPZ0}).  \end{remark}

\begin{remark} Sometimes we want to indicate the dependence of the solution on the parameters by writing $h(t,x;\lambda,\nu,\sqrt{\nu})$.  It is interesting to take $\lambda =1$, $\sigma^2=\nu$ because  $h(t,x;1,\nu,\sqrt{\nu})=\nu h(\nu^{-3}t, \nu^{-2}x; 1,1,1)$ corresponds to the 
solution of $\partial_t h = -  (\partial_xh)^2 + \nu \partial_x^2 h + \sqrt{\nu }\dot W
$ and note that
the result implies that there is a $\nu_0>0$ and fixed $0< c_1\le c_2 <\infty $ such that 
$ c_1t^{2/3} \le \Var (h(t,0;1,\nu,\sqrt{\nu}))\le c_2t^{2/3}$ for $t\ge 1$ for all $\nu\le \nu_0$.
\end{remark}

Now we turn to results about the correlations of $u(t,x)$.   Throughout we will assume that it is in equilibrium with initial data white noise with
variance $\nu^{-1}\sigma^2$. 
 By definition 
$
\langle\varphi, u(t)\rangle=-\int \varphi'(x)h(t,x)dx
$ for $\varphi\in C_c^\infty(\mathbb R)$. 
 For each fixed $t>0$, define a bilinear functional on $C_c^\infty(\mathbb R)\times C_c^\infty(\mathbb R)$
\[
B_t(\varphi,\psi) = E[\langle\varphi, u(t)\rangle\langle\psi, u(0)\rangle].
\] 
The following proposition provides us with our definition of the space-time correlation measure of the stochastic
Burgers equation.  

\begin{prop} \label{pr} For each $t>0$ there is a unique probability measure $S(t,dx)$ on ${\mathbb R}$ such 
that for $\varphi,\psi \in C_c^\infty(\mathbb R)$
\begin{equation}\label{preq}
B_t(\varphi,\psi) = \int_\R  \biggl[\frac{1}{2} \int_\R \varphi\Bigl( \frac{y+x}{2}\Bigr) \psi\Bigl( \frac{y-x}{2}\Bigr)
\,dy \biggr] S(t,dx).
\end{equation}
$S(t,\cdot)$ is symmetric: $S(t,A)= S(t,-A)$ for Borel sets $A$ where $-A=\{-x: x\in A\}$. 
The connection with the process $h$ is that 
\be  \Var(h(t,x)) ={x}  +2 \int_{{x}}^\infty  (z-{x})\,S(t,dz)  
\label{hSa}\ee
for $x\in\R$ and $t>0$.
\end{prop}

The route to constructing $S(t)$ and proving \eqref{hSa} is somewhat
circuitous. The measure $S(t)$ is constructed as a weak limit
from the rescaled correlations of a particle process.
Then we show that, in the sense of distributions,
\begin{equation}\label{td}
\tfrac12\partial_x^2 \Var (h(t,x)) = S(t,dx).
\end{equation}
Finally, after studying solutions of the stochastic heat equation,
we can deduce \eqref{hSa}.  Here are the bounds on $S(t)$.

\begin{theorem}\label{theoremu} Let  $u(t)=\partial_x h(t)$ be the distributional derivative of the Hopf-Cole solution  {\rm (\ref{def})} of {\rm (\ref{KPZ0})}, in equilibrium with initial data white noise with  variance $\nu^{-1}\sigma^2$ and let $S(t,dx)$ be the space-time correlation measure defined through {\rm (\ref{preq})}.  With the same constants as in Theorem \ref{theoremh}, 
for $
t\ge c_0 
$
we have 
 \begin{equation} \label{bdsonu}
C_1 t^{2m/3} \le \int |x|^{m} S(t, dx)  \le C_2 t^{2m/3}.
\end{equation} The upper bound holds for $1\le m<3$. The lower bound holds for all $m\ge 1$.
 In particular, the {\rm diffusivity} 
\begin{equation}
D(t) = \frac{1}{t} \int x^2 S(t,dx) 
\end{equation}
satisfies
\begin{equation} 
 C_1 t^{1/3} \le D(t) \le C_2 t^{1/3}.
\end{equation}
\end{theorem}

We turn to proofs. The first issue is to develop the connection with
the exclusion process.

\section{Weakly asymmetric simple exclusion}
\label{weakly}
\setcounter{equation}{0}
We consider nearest neighbour (i.e.~simple) exclusion on $\mathbb Z$ (ASEP) with particles attempting jumps to the right at rate $p=1/2$ and to the 
left at rate $q=1/2 +\eps^{1/2}$ with $\eps\in(0,1/4)$.
This is a system of continuous time random walks jumping to the right at rate $p$ and to the 
left at rate $q$, with the rule that jumps to already occupied sites are not realized.
Hence the occupation variable can be taken to be $\eta(t,x)=1$ or $0$ depending 
on whether or not there is  a particle at $x\in\mathbb Z$ at time $t$.  One of the most important properties of this
system is that it preserves Bernoulli product measures with any density $\rho\in [0,1]$. Here we take $\rho =1/2$,
i.e. we take as initial 
configuration  independent Bernoulli $\{\eta(0,x)\}$, $x\in \mathbb Z$,  with density $1/2$.

  Let $\hat\eta=  2\eta-1$ and define the height function
\begin{equation}
\zeta_\eps(t,x)=\left\{\begin{array}{ll}
\sum_{0<y\le x} \hat\eta(t,y) - 2N(t,0) , & x>0, \\[6pt]
-2N(t,0), & x=0, \\[4pt]
-\sum_{x<y\le 0} \hat\eta(t,y) - 2N(t,0) , & x<0,\end{array}\right.
\label{psheight}\end{equation}
where $N(t,0)$ is the current across the bond   $(0,1)$
up to time $t$, i.e.\ the number of particles that jump from $0$ to $1$
minus the number of particles that  jump from $1$ to $0$ in the time
interval $[0,t]$.  
  For $x\in  \mathbb R$ 
and $t\ge  0$ let $h_\eps(t,x)$ denote the rescaled height function;
\begin{equation}
h_\eps(t,x) = \eps^{1/2}\bigl( \zeta_\eps(\eps^{-2}t,[\eps^{-1}x]) - v_\eps t\bigr) 
\label{scaledhgt}\end{equation}
where
$
 v_\eps= \frac12 \eps^{-3/2} - \frac{1}{4!} \eps^{-1/2}
$
 and the closest integer $[x]$ is given by \begin{equation}[x] = \lfloor x+\tfrac12\, \rfloor.\label{closestinteger}\end{equation}

 
 We think of $h_\eps$ as an element of the space $D([0,\infty); D_u(\mathbb R))$ where $D$ refers to right continuous paths with left limits.  $D_u(\R)$ indicates
 that in space
  these functions are  equipped with  the topology of uniform convergence on compact sets. 
Because the discontinuities of $h_\eps(t,\cdot)$ are restricted to
 $\eps(1/2+\Z)$,
it is measurable as a $D_u(\R)$-valued random function (see Sec. 18 of 
\cite{Bill}.)   Since the jumps of $h_\eps(t,\cdot)$ are uniformly small, 
local uniform convergence works for us just as well the standard 
  Shorohod topology.  
  The  probability distribution of the process $h_\eps$ on $D([0,\infty); D_u(\mathbb R))$ will be 
  denoted ${\mathcal P}_\eps$.
  
\begin{prop} {\rm\cite{BG}}   As $\eps\searrow 0$, the distributions 
${\mathcal P}_\eps $ converge weakly to  ${\mathcal P}$,  
the distribution of the Hopf-Cole solution $h$  {\rm (\ref{def})} of 
{\rm (\ref{KPZ0})} with $\lambda=1/2$, $\nu=1/2$,  and 
$\sigma=1$.
 \end{prop}
 
\begin{proof}  This was proved in  \cite{BG} using  the slightly different height function $ \zeta^{\text{BG}}(t,x)$ related to ours by
\be   \zeta^{\text{BG}}(t,x)=\zeta_\eps(t,x)-2(1-\eta(0,0)).  \label{heightconn}\ee
The result follows for $\zeta_\eps(t,x)$ because the difference is bounded.
Note also that 
\cite{BG} makes the height functions continuous in space by linear interpolation and uses the 
smaller path space $D([0,\infty); C(\mathbb R))$.   This makes no difference because
$h^{\text{BG}}_\eps(t,x)$ and its continuous version are uniformly within distance $\eps^{1/2}$
of each other.  
\end{proof} 
 The rescaled velocity field is 
\begin{equation}
u_\eps(t,x) = \eps^{-1/2} \hat\eta( \eps^{-2}t, [ \eps^{-1} x] )
\label{defS} \end{equation}
with $[\cdot]$ defined by (\ref{closestinteger}).
The  rescaled space-time correlation functions are given by 
\begin{equation}
S_\eps(t,x) = E_\eps[ u_\eps (t,x) u_\eps(0,0) ]. 
\end{equation}  
For $x\in\R$ let us define a discrete Laplacian   by
\begin{equation}
\label{delta}
\Delta_\epsilon f(x) = \tfrac12 \epsilon^{-2} ( f(x+\epsilon) -2f(x) + f(x-\epsilon) ) 
 \end{equation}
and a discrete absolute value   in terms of the closest integer function by \begin{equation}
|x|_\epsilon = \abs{\eps[{\eps^{-1}x}]}. 
\end{equation}
We begin by building on some well-known  properties.   The assumption
that density is $1/2$ is used repeatedly.  

\begin{prop} \label{prop1} {\it 1.} For a fixed $t>0$,   $S_\eps(t,x)$ is a probability density on $\R$ 
and symmetric in  $x$  {except at $x\in  \epsilon(1/2+ \mathbb Z)$}.

\noindent{\it 2.}   $S_\eps(t,x) =\Delta_\eps\Var (h_\eps(t,x))$. 
 $\Var (h_\eps(t,x))-|x|_\epsilon$ is symmetric in  $x\notin  \epsilon(1/2+ \mathbb Z)$, non-decreasing in $|x|_\epsilon$,  nonnegative, and for each fixed $(\epsilon, t)$ 
has exponentially decaying  tails in $x$. 

\noindent{\it 3.}  For  $1\le m<\infty$,  
\be \int_\R |x|_\epsilon^m S_\eps(t,x) dx = \int_\R \Delta_\eps (|x|_\epsilon^m)\, [ \Var (h_\eps(t,x)) - |x|_\epsilon ] \,dx.\label{v-|x|}\ee

\end{prop}

\begin{proof} 
To see that $S_\eps(t,x)$ is a probability density, use the well-known connection
with the   second class 
particle:  
\be    S_\eps(t,x) = \eps^{-1} \Pv^{1/2}_\eps \{{\sc}(\eps^{-2}t)=[\eps^{-1}x]\}.   \label{Ssc}\ee
Here $\Pv^{1/2}_\eps$ is the coupling measure of two ASEP's that start   with
one discrepancy at the origin and Bernoulli(1/2) occupations elsewhere, and 
$\sc(\cdot)$ is the position of the second class particle.  This setting is discussed
  in Section \ref{second}.  A proof of \eqref{Ssc} can be found for example
in \cite{bala-sepp-08ALEA}.  
Symmetry of $S_\eps(t,x)$ can be  seen from the definition \eqref{defS} 
and  the fact that  
\be  
\wt\eta(t,k)=1-\eta(t,-k) \label{etahat}\ee
defines an ASEP $\wt\eta$ equal in distribution to $\eta$. 
 See \cite{PS} for the    explicit computation that proves   {\it 2}.  
 
We now work towards {\it 3}.  Let $N(t,x)$ denote the current across the bond between site $x$ and $x+1$ up to time $t$.
 We start by checking  that 
  \be \label{sym}
\Cov\Bigl[ N(t,0), \sum_{y=-x+1}^x\eta(t,y)\Bigr]=0.
\ee
With $\wt\eta$ as in \eqref{etahat},   an $\wt\eta$-particle jump from $x$ to $y$ 
is the same as an $\eta$-hole jump from $-x$ to $-y$.  Hence 
  $\wt N(t,0)=N(t,-1)$.  By the distributional equality $\wt\eta\overset{d}=\eta$, 
$\Cov\Bigl[ N(t,0), \sum_{y=-x+1}^x\eta(t,y)\Bigr]=
\Cov\Bigl[ \wt N(t,0), \sum_{y=-x+1}^x\wt\eta(t,y)\Bigr]=\Cov\Bigl[ N(t,-1), -\,\sum_{y=-x}^{x-1}\eta(t,y)\Bigr]
=
-\,\Cov\Bigl[ N(t,0), \sum_{y=-x+1}^x\eta(t,y)\Bigr]
$
and \eqref{sym} is verified. 

Combining  Lemma 3 of \cite{QV}   with \eqref{sym} gives  
 \begin{equation}\label{varminusx}
 \Var (\zeta_\eps(t,x))-|x| = 4 \,\Cov (N(t,0), N(t,x)), \quad x\in\Z.  
\end{equation}
The right-hand side of \eqref{varminusx} is symmetric in $x$  by
invariance under  spatial translations.   

Next, note that by the finite range of ASEP, 
  for fixed $\eps>0$ and  $t\ge 0$ there exist  $C_1<\infty$, $C_2>0$ such that 
\begin{equation}\label{cov2}
\bigl \lvert \Cov (N(t,0), N(t,x)) \bigr\rvert  \le C_1\exp\{ -C_2 |x| \}.
\end{equation}
(see Lemma 4 of \cite{QV} for details.)

We next argue the   nonnegativity of \eqref{varminusx}.  
By symmetry it suffices to consider $x\in\Z_+$.   For $x>0$, $ v(x)=\Var (\zeta_\eps(t,x))$ 
and $v(x)-\abs{x}$ have the same discrete Laplacian.  By  \eqref{varminusx} and \eqref{cov2} 
both $v(x)-\abs{x}$ and $\Delta_1 v(x)$ decay exponentially, and, by 
  {\it 1} and {\it 2},  $\Delta_1 v(x)\ge 0$.  
Then for $x\in\Z_+$
\be
v(x)-\abs{x} =\sum_{k=x}^\infty \sum_{\ell= k}^\infty  2\Delta_1 v(\ell+1) > 0.  \label{vLapl}\ee
This also shows that $v(x)-\abs{x}$ is strictly decreasing for $x\in\Z_+$, and thereby 
\eqref{varminusx} is strictly decreasing in $\abs{x}$. 
  
Putting the scaling into \eqref{varminusx} gives
\be\begin{aligned}
   \Var (h_\eps(t,x))-|x|_\epsilon &= 4\eps \Cov \bigl(N(\eps^{-2}t,0), 
   N(\eps^{-2}t,[{\eps^{-1}x}])\bigr).  \end{aligned}\label{varminusx2}\ee
This proves (3). 
To prove (4), start with the observation 
\[  \int_{-N}^N |x|_\epsilon^m \,\Delta_\eps (|x|_\epsilon)\, dx=0.  \]
Then by (2) and by integration by parts (that is, by shifting the integration variable), 
 \begin{align*}
 \int_{-N}^N |x|_\epsilon^m S_\eps(t,x) dx &= 
  \int_{-N}^N |x|_\epsilon^m \,\Delta_\eps [ \Var (h_\eps(t,x)) - |x|_\epsilon ]\, dx \\
&= \int_{-N}^N \Delta_\eps (|x|_\epsilon^m)\, [ \Var (h_\eps(t,x)) - |x|_\epsilon ]\, dx + B_{\eps, t, m}(N)
\end{align*}
%
where  $B_{\eps, t, m}(N)$ are sums of integrals 
of $ \abs{x\pm\eps}_\epsilon^m[\Var (h_\eps(t,x)) - |x|_\epsilon]$ 
over intervals of length $\eps$ around $\pm N$.     By {\it 3}
 these are  exponentially small in $N$ 
as   $\eps$, $t$ and $m$ are fixed.  Taking $N\to \infty$  gives {\it 4}.
 \end{proof}

 The key technical estimate which will be proved in  Section \ref{second} is

\begin{theorem}  With the same constants as in Theorem \ref{theoremh}, for all $0<\epsilon < 1/4$,  $1\le m<3$, and $t\ge c_0$, 
\be C_1t^{2m/3} \le  \int |x|_\epsilon^mS_\eps(t,x)dx \le C_2t^{2m/3}.   
\label{Qmoma}\ee\label{Qmomthma}\end{theorem} 


\begin{cor}\label{cor1}  
{\it 1.}  For  $0<\epsilon < 1/4$ and   $t\ge c_0$, 
  \be C_1t^{2/3}\le \Var (h_\eps(t,0))\le C_2t^{2/3}. \label{varheps}\ee 

\noindent{\it 2.}  For each  $t>0$, the family of probability measures 
$\{S_\eps(t,x)dx\}_{0<\eps< 1/4}$ is tight.
\end{cor}
\begin{proof} Part {\it 1} follows from Theorem \ref{Qmomthma} and case  $m=1$ of
 {\it 4} of 
  Proposition \ref{prop1} because $\Delta_\eps (|x|_\epsilon)=\eps^{-1}$ for 
  $x\in[-\eps/2,\eps/2)$ and vanishes elsewhere. 

For $t\ge c_0$  tightness of $\{S_\eps(t,x)dx\}_{0<\eps< 1/4}$ 
 follows from the upper bound in \eqref{Qmoma}.  For 
$0<t<c_0$ recall the second class particle connection \eqref{Ssc}. 
Proposition 4 in \cite{QV} proves that the second moment 
\[   \int |x|_\epsilon^2 S_\eps(t,x)dx = 
\eps^2 \Ev\bigl( \abs{{\sc}(\eps^{-2}t)}^2\, \bigr) \]
is monotone nondecreasing in $t$.  Thus the large-$t$ bound gives
the tightness for all $t>0$.  
\end{proof}

\section{Proofs of the main results}
\label{proofs}
\setcounter{equation}{0}

As a preliminary point we discuss the regularity of 
$\Var(h(t,x))$.  
The control comes from  the weak limit $h_\eps\to h$.  We have
\[
\Var(h_\eps(t,0))=\int \abs{x}_\eps\, S_\eps(t,x)\,dx 
\le \biggl( \int \abs{x}_\eps^2\, S_\eps(t,x)\,dx   \biggr)^{1/2}.
\] 
As mentioned in the proof of Corollary \ref{cor1}, this last 
quantity is nondecreasing in $t$.  Consequently by the upper
bound in \eqref{Qmoma} and the i.i.d.\ mean zero spatial 
increments of $h_\eps(t,x)$ [see \eqref{psheight}] we conclude
that  $\Var(h_\eps(x,t))$ 
is locally bounded as a function of $(t,x)$, 
uniformly in $\eps>0$.  By the weak limit
$\Var(h(t,x))\le \varliminf_{\eps\to 0}\Var(h_\eps(t,x))$
and so    $\Var(h(x,t))$ 
is locally bounded.   
 The $x$-symmetry of $\Var (h(t,x))$ follows from
part  {\it 6} of Proposition \ref{thm1}, 
or from the weak limit $h_\eps\to h$ 
 and the distributional symmetry of  $\zeta_\eps(t, \,\cdot)$.
For any fixed $x_0$, $h(t,x)-h(t,x_0)$ is a Brownian
motion in $x$ and hence the continuity of $x\mapsto \Var (h(t,x))$. 
By studying the stochastic heat equation we  prove in the Appendix  that 
\be \Var (h(t,x))-|x| \to 0 \quad\text{ as $\abs{x}\to\infty$.}
\label{auxh0}\ee

We turn to proving the main results.

\begin{proof}[Proof of (\ref{preq})]  From the definitions we have for test functions $\varphi, \psi \in C_c^\infty(\R)$
\be\begin{aligned}
&E[ \langle \varphi',h_\eps(t) \rangle \langle \psi', h_\eps(0)\rangle ]  = 
E[ \langle \varphi,u_\eps(t) \rangle \langle \psi, u_\eps(0)\rangle ] +O(\eps) \\[4pt] 
&\qquad=  \frac{1}{2} \int  \Bigl[ \,\int \varphi \Big( \frac{y+x}{2} \Bigr) \psi \Big( \frac{y-x}{2} \Bigr)dy \Bigr] S_\eps(t,x)\,dx
 +O(\eps).
\end{aligned}\label{preqi}
\ee
Let $S(t,dx)$ denote a weak limit point of $S_\eps(t,x)dx$
as $\eps\searrow 0$.
Taking the limit in \eqref{preqi}, 
  the last expression becomes the right-hand side of (\ref{preq}). 

Convergence of the first expectation in   \eqref{preqi} to the left-hand side of \eqref{preq}
follows from the weak convergence $h_\eps\to h$ with the following additional justification.   Since the limit process $h$
is continuous (in time),  the pair  $(h_\eps(0), h_\eps(t))$ converges weakly to $(h(0),h(t))$. 
Integration against a $C_c^\infty(\R)$ function is a continuous function on $D_u(\R)$. 
By an application of   uniform integrability  and Schwarz inequality,  for 
  the claimed convergence it is enough to show the $L^1$ boundedness of 
 $\abs{\langle \varphi',h_\eps(t) \rangle}^4$.  To see this   we  first 
 transform the inner product (this is the beginning of the computation that 
 one performs to check \eqref{preqi}):
\begin{align*}
 \langle \varphi',h_\eps(t) \rangle 
 &= \eps^{1/2}\sum_{k\in\Z}  \zeta_\eps(\eps^{-2}t, k) 
 \bigl[  \varphi((k+\tfrac12)\eps) - \varphi((k-\tfrac12)\eps)\bigr]\\
 &= -\,\eps^{1/2}\sum_{k\in\Z} \varphi((k-\tfrac12)\eps) \hat\eta(\eps^{-2}t, k).
\end{align*}
(The constant term on the right-hand side of definition \eqref{scaledhgt} vanishes since
we are integrating the height function against a derivative.) 
This is a sum of independent  mean zero random variables, and 
\begin{equation}
E \abs{\langle \varphi',h_\eps(t) \rangle}^4 \le C\eps^2 \sum_k  \varphi((k-\tfrac12)\eps)^4
+ C\eps^2 \sum_{k,\ell}  \varphi((k-\tfrac12)\eps)^2  \varphi((\ell-\tfrac12)\eps)^2 \nonumber \end{equation}
which is bounded uniformly in $\epsilon$.
\end{proof}

\begin{proof}[Proof of Theorem \ref{theoremu} and the upper bound of Theorem \ref{theoremh}]  Note first of all that from
part  {\it 6} of Proposition \ref{thm1}, it suffices to prove all results with $\lambda=1/2$, $\nu=1/2$,
$\sigma=1$.  

 The upper bounds of \eqref{bdsonh} and (\ref{bdsonu}) follow  from the 
weak convergence and from the 
upper bounds  in \eqref{Qmoma} and  in {\it 1} of Corollary \ref{cor1}.
 
Let $1\le m<3$. For the upper bound of (\ref{bdsonh1}) we collect these ingredients:
  Inequality 
$\Var (h_\eps(t,x)) - |x|_\eps  \ge 0$
  from  {\it 3} of Proposition \ref{prop1},  
  the fact that under the weak limit 
 \begin{equation}\liminf_{\eps\searrow 0} \bigl[ \Var (h_\eps( t, x)) -\abs{x}_\eps\bigr]
 \ge \Var (h(t,x))-\abs{x}, 
 \end{equation} 
and  for $x\ne 0$,  $\Delta_\eps(\abs{x}_\eps^m)\to m(m-1)\abs{x}^{m-2}/2$.  
Combine the upper bound   in
\eqref{Qmoma}  with identity \eqref{v-|x|},   let $\eps\searrow 0$ in  \eqref{v-|x|} and 
use  Fatou's Lemma.
 
To prove the lower bound of (\ref{bdsonu}), let $t>c_0$ be fixed and choose a non-negative smooth function $f(x)$ with compact support  such that $f(x) \ge |x|^m$ for $|x|\le At^{2/3} $.
We have \[\int f(x) S(t,dx) = \lim_{\eps\searrow 0} \int f(x) S_\eps(t,x)dx\] and furthermore
\[\int f(x) S_\eps(t,x)dx\ge \int_{|x|\le At^{2/3} } |x|^mS_\eps(t,x)dx .
\]
 Choose $\delta>0$ such that $m+\delta<3$.  By Chebyshev's inequality and Theorem \ref{Qmomthma},
\[
\int_{|x|\ge At^{2/3} } |x|^mS_\eps(t,x)dx \le A^{-\delta} t^{-2\delta/3} \int |x|^{m+\delta} S_\eps(t,x)dx \le A^{-\delta} Ct^{2m/3}.
\]
Hence,  for appropriately chosen $A$,
\[\int f(x) S_\eps(t,x)dx\ge  ( C^{-1}/2  )t^{2m/3}.
\]
Since this is true for all such $f$, we conclude that the lower bound of (\ref{bdsonu}) holds.  \end{proof}

\begin{proof}[Proof of (\ref{td})] 
 We are unable to do this by direct approximation 
 due to lack of control of moments of $h_\eps(t,x)$ higher than $2$.  
By direct calculation  $
  E [h_\eps(t,x)]
    = t/4! 
  $
 and we can take the $\eps\searrow 0$ limit by uniform
 integrability that follows from
the boundedness of $\Var(h_\eps(t,x))$ argued in the beginning of
this section. 
  Consequently \begin{equation}
  E [h(t,x)]
    = t/4! 
  .\end{equation}  From   
    \[\Var (h(t,x)) = E[( h(t,x)-h(0,0)-t/4!)^2]\]  
we deduce 
\begin{align*} \Delta_\delta \Var (h(t,x))&= \delta^{-2} 
E[ \left( {\scriptstyle{\frac{ h(t,x+\delta) +h(t,x)  }{2} }}-h(0,0) -t /4!\right) ( h(t, x+\delta)-h(t,x) ) \\[3pt] 
&\quad  - \left(  {\scriptstyle{\frac{ h(t,x) +h(t,x-\delta)  }{2} }}-h(0,0) -t/4! \right) ( h(t, x)-h(t,x-\delta) )].  \end{align*}
Since increments are mean zero and stationary in space  (part {\it 5} of Proposition \ref{thm1}), the latter is equal to
\[
\delta^{-2} E[ \left( h(0,\delta-x) -h(0,-x)  \right) ( h(t, \delta)-h(t,0) )].\]
Define the ``tent function'' $\varphi_\delta(x) = (\delta^{-1}- \delta^{-2}|x|)1_{|x|\le \delta} $.  We have shown that
\[
\Delta_\delta\Var(h(t,x))= \langle \tau_x \varphi_\delta, S(t) \rangle
\]
where the angle brackets denote integration and $ \tau_x \varphi_\delta(y)=   \varphi_\delta(y-x)$ and we used the definition (\ref{preq}) of $S(t,dx)$.
  Integrating against a test function $\psi\in C_c^\infty(\R)$ gives
\[ \int \Delta_\delta\psi (x)  \Var (h(t,x))dx = \langle \psi\ast\varphi_\delta, S(t) \rangle\] with $\ast$ denoting  convolution. Let $\delta\searrow 0$. Since $\Var(h(t,x))$ is locally bounded 
we can take the limit on the left. In the limit we   obtain
\[ \tfrac12 \int \partial_x^2\psi(x)  \Var(h(t,x))dx = \langle \psi, S(t) \rangle.  \] 
This completes the proof of Proposition \ref{pr}. \end{proof}

We now complete  the proof of
Proposition \ref{pr} and 
 Theorem \ref{theoremh} with the following

\begin{prop}  For $x\in\R$ and $t>0$, 
\be  \Var(h(t,x)) ={x}  +2 \int_{{x}}^\infty  (z-{x})\,S(t,dz). \label{hS}\ee
\label{varh-prop-9}\end{prop}  

This proposition implies the remaining parts of Theorem \ref{theoremh}
because  symmetry implies
\[  \Var(h(t,x)) =\abs{x}  +2 \int_{\abs{x}}^\infty  (z-\abs{x})\,S(t,dz).
\]
 From this follow $\Var(h(t,x)) -\abs{x}\ge 0$ and 
 the identities 
\[   \Var(h(t,0)) =  \int_\R  \abs{x}\,S(t,dx)  \]
and 
for $1<m<3$  
  \[m(m-1) \int |x|^{m-2} [\Var (h(t,x))-|x|]dx = 2\int |x|^m S(t,dx).  \]
Then we can apply the bounds from \eqref{bdsonu}. 

\begin{proof}[Proof of Proposition \ref{varh-prop-9}]  First a lemma. 

\begin{lemma}  Suppose $v$ is a continuous, symmetric function on $\R$
and in the sense of distributions  $v''=2\mu$ for a symmetric probability
measure $\mu$ on $\R$.  Assume that $\int_\R\abs{x}\,d\mu<\infty$.  Then 
\be v(x)= \abs{x} + v(0)-\int_\R \abs{z}\,\mu(dz) + 2\int_{\abs{x}}^\infty (z-\abs{x}) \mu(dz).
\label{v''identity}\ee
\label{v''lemma}\end{lemma}

\begin{proof}  Suppose first that $v''/2$ is a continuous probability density. Then
\[  g(x)=\int_x^\infty (z-x) v''(z)\,dz \]
satisfies $g''=v''$ and thereby 
\[  v(x)= ax+b+ \int_x^\infty (z-x) v''(z)\,dz \]
for constants $a$, $b$.  From symmetry deduce $a=1$.  Taking $x=0$ 
identifies  $b=v(0)-\tfrac12 \int_\infty^\infty \abs{z}v''(z)\,dz.$
Now \eqref{v''identity} holds for smooth $v$.  Take a symmetric compactly
supported smooth approximate identity $\{\phi_\delta\}_{\delta>0}$, apply 
 \eqref{v''identity} to $\phi_\delta*v$ and let $\delta\searrow 0$. 
\end{proof}

Continuing the proof of Proposition \ref{varh-prop-9}, 
apply \eqref{v''identity} to $v(x)=\Var(h(t,x))$ to get  
 \be\begin{aligned}
\Var(h(t,x))&=\abs{x}+\biggl( \Var(h(t,0)) - \int_{\R} \abs{z}\,S(t,dz)\biggr) \\
&\qquad +2 \int_{\abs{x}}^\infty  (z-\abs{x})\,S(t,dz).  
\end{aligned} \label{v-ident}\ee From the Appendix we get   $\Var(h(t,x))-\abs{x}\to 0$.
Combining this with above gives  first 
\be   \Var(h(t,0)) = \int_{\R} \abs{z}\,S(t,dz) \label{v0-ident}\ee
and then  
\be\begin{aligned}
\Var(h(t,x))&=\abs{x}  + 2\int_{\abs{x}}^\infty  (z-\abs{x})\,S(t,dz)\\ 
&={x}  +2 \int_{{x}}^\infty  (z-{x})\,S(t,dz).  \qedhere\\ \end{aligned} \label{v-ident2}\ee 
\end{proof}

\section{Second class particle estimate}
\label{second}

In this section we prove the key estimate for the moment of a second class
particle.  The context is the asymmetric simple exclusion process (ASEP) jumping to the right with rate $p=1/2$ and
to the left with rate $q=1/2+\ep^{1/2}$. Throughout $\ep\in(0,1/4)$, with
the real interest being the limit $\ep\searrow 0$.
 Probabilities associated to this process
are denoted by $P_\ep^\rho$ when the process is stationary with Bernoulli $\rho$
occupations.  The macroscopic flux function  is 
$H_\ep(\rho)=-\ep^{1/2}\rho(1-\rho)$ and the characteristic speed 
$V^\rho_\ep  = H_\ep'(\rho)=    -\ep^{1/2}(1-2\rho)$.

Let $\Pv^\rho_\ep$ denote the probability measure of the basic coupling of two processes
$\zeta^-(t)\le\zeta(t)$ with this initial configuration:  $\zeta^-(0,0)=0<1=\zeta(0,0)$, and
for  $x\ne 0$, $\zeta^-(0,x)=\zeta(0,x)$ have mean $\rho$ and they are independent 
across the sites $x$.
 Let ${\sc}(t)$ denote the position of the discrepancy between $\zeta^-(t)$ and $\zeta(t)$, 
in other words, the position of the   second class  particle 
started at the origin.   The mean speed of the second class  particle is
the characteristic speed (Corollary 2.5  in \cite{bala-sepp-07JSP} or 
Theorem 2.1 in \cite{bala-sepp-08ALEA}):
\be  \Ev^\rho_\ep  [{\sc}(t)] = V^\rho_\ep t.   \label{Qspeed}\ee

 From (\ref{Ssc}), Theorem \ref{Qmomthma} is equivalent to the case $\rho=1/2$ of the following theorem. 

\begin{theorem}  With the same constants as in Theorem \ref{theoremh}, for all $0<\ep<1/4$,  $1\le m<3$, and $t\ge c_0\ep^{-2}$, 
\be C_1\ep^{m/3}t^{2m/3} \le  \Ev^\rho_\ep  \bigl[\,\abs{\sc(t)-V^\rho_\eps t}^m \,\bigr]\le C_2\ep^{m/3}t^{2m/3}.   
\label{Qmom}\ee\label{Qmomthm}\end{theorem} 

The remainder of the section proves Theorem \ref{Qmomthm}, with separate subsections
for the upper and lower bound. 




\subsection{Proof of the upper bound for Theorem \ref{Qmomthm}}  


\begin{lemma}   Let $B\in(0,\infty)$.  
 There exists 
$C\in(0,\infty)$  and $c_1(B)\in(0,\infty)$
    such that the following bounds hold for
all   $0< \rho< 1$, $u\ge 1$, $0<\ep<1/4$, and  
  $t\ge c_1(B)\ep^{-1/2}$.   
  
  {\rm (i)}  For $B\ep^{1/3} t^{2/3}\le  u\le 20t/3$, \be
\Pv^\rho_\ep(|{\sc}(t)- V^\rho_\ep t|\ge u) 
 \le 
 C\ep t^2u^{-4}\Ev^\rho_\ep\lvert {\sc}(t)-V^\rho_\ep t\rvert+ C\ep t^2u^{-3} +e^{-u^2/Ct}.
\label{eq:UBge1/2} \ee 

{\rm (ii)} For $u\ge 20 t/3$,
 \be
\Pv^\rho_\ep(|{\sc}(t)- V^\rho_\ep t|\ge u) 
 \le 
e^{-u/C}. 
\label{eq:UBge2/2}\ee \label{lm:UB1}
\end{lemma}

\begin{proof}  First we obtain the bounds for $\Pv^\rho_\ep({\sc}(t)\le V^\rho_\ep t - u)$.   By an adjustment
of  the constant $C$ we can 
 assume   that  
$u$ is a positive integer.  
Fix a density
 $0< \rho <1$ and let  $\lambda\in(0,\rho)$. 
Consider a basic coupling of three ASEP's
$\zeta\ge\zeta^-\ge \eta$ with this initial configuration: 
 
(a) 
Initially $\{\zeta(0,x):x\ne 0\}$ are i.i.d.\ Bernoulli($\rho$) 
 and $\zeta(0,0)=1$.

(b) 
Initially $\zeta^-(0,x)=\zeta(0,x)-\delta_0(x)$.   

(c) Initially
  $\{\eta(0,x):x\ne 0\}$ 
are i.i.d.\ Bernoulli($\lambda$) and $\eta(0,0)=0$. 
 The   coupling of the 
initial occupations is 
such that $\zeta(0,x)\ge\eta(0,x)$ for all $x\ne 0$.

Recall that  basic coupling means that  the processes
share common  Poisson clocks. 

Let $\sc(t)$ be the position of the single second class 
particle between $\zeta(t)$ and $\zeta^-(t)$, initially
at the origin.  Let $\{X_i(t):i\in\Z\}$ be the positions
of the $\zeta-\eta$ second class particles,  
labeled so that initially 
\[
\dotsm<X_{-2}(0)<X_{-1}(0)<X_0(0)=0<X_1(0)<X_2(0)<\dotsm 
\]
Let these second class particles preserve their labels in the
dynamics and stay ordered.  Thus the $\zeta(t)$ configuration
consists of first class particles (the $\eta(t)$ process)
and second class particles (the $X_j(t)$'s). 
Let $\Pv_\ep$ denote the joint probability distribution of these
coupled processes. The marginal distribution
of $(\zeta,\zeta^-,{\sc})$ under $\Pv_\ep$ is the same as   under $\Pv_\ep^\rho$. 

For $x\in\Z$,  
 $J^\zeta_x(t)$ is  the net left-to-right particle current in the 
$\zeta$ process across the space-time line segment from  point  
$(1/2,0)$ to $(x+1/2,t)$.  Similarly 
 $J^\eta_x(t)$  in the $\eta$ process, 
and  $J^{\zeta-\eta}_x(t)$ is  the net current of
second class particles.  
Current 
in the $\zeta$ process is a sum of  the first class particle
current and the second class particle current:  
$
J^\zeta_x(t)= J^\eta_x(t)+J^{\zeta-\eta}_x(t). $

Basic coupling preserves
 $\sc(t)\in\{X_j(t)\}$.  Define the label $m(t)$ by
$\sc(t)=X_{m(t)}(t)$ with initial value $m(0)=0$.  The label $m(t)$ performs a walk
on the labels of the $\{X_j\}$ with rates $p$ to the left and $q$ to the right, 
but jumps permitted only when  $X_j$ particles are  adjacent. 
Through a  comparison with a reversible walk, 
Lemma 5.2 in \cite{bala-sepp-08ALEA} gives the bound 
\be
\Pv_\ep(m(t)\le -k) \le \exp\{-\ep^{1/2} k\} \quad\text{for all $t\ge 0$ and $k\ge 0$.}
\label{mQ}\ee
 
 To get the first step of the estimation,  note that  if $\sc(t)\le V^\rho_\ep t -u$ and 
$m(t)>-k$, then $X_{-k}(t)<\fl{V^\rho_\ep t}-u$.   Then among the  $\zeta-\eta$
particles only $X_{-k+1},\dotsc, X_{0}$ could have crossed from the left side of $1/2$
to the right side of $ \fl{V^\rho_\ep t} -u+1/2$ during time $(0,t]$.  Thereby 
  $J^{\zeta-\eta}_{\fl{V^\rho_\ep t}-u}(t)\le k$, and by an appeal to  \eqref{mQ} we have 
 \begin{align}
&\Pv_\ep(\sc(t)\le V^\rho_\ep t -u\} \le \Pv_\ep(m(t)\le -k)
+ \Pv_\ep(  J^\zeta_{\fl{V^\rho_\ep t}-u}(t) \; - \; J^\eta_{\fl{V^\rho_\ep t}-u}(t)
\le k) \nn \\
&\qquad\le  \exp\{-\ep^{1/2} k\}  
+ \Pv_\ep( J^\zeta_{\fl{V^\rho_\ep t}-u}(t) \; - \; J^\eta_{\fl{V^\rho_\ep t}-u}(t)
\le k).
\label{line2}\end{align}

We work on the second probability on line
\eqref{line2}.   In  stationary density $\rho$ 
$E^\rho_\ep[J_x(t)]=H_\ep(\rho) t-\rho x$ for $x\in\Z$. 
 Process $\zeta$ can be coupled
with a stationary density $\rho$ process $\zeta^{(\rho)}$ 
with at most one discrepancy.  In this coupling
$\lvert J^\zeta_{x}(t)-J^{\zeta^{(\rho)}}_{x}(t)\rvert\le 1$ and so  
 we can use expectations of  
stationary processes at the expense of small errors.  
  Let
$c_1$ below be a constant that absorbs the errors
from using means of stationary processes  
 and from ignoring  integer parts. It satisfies 
$\abs{c_1}\le 3$.  
\begin{align}
\Ev_\ep J^\zeta_{\fl{V^\rho_\ep t}-u}(t) - \Ev_\ep J^\eta_{\fl{V^\rho_\ep t}-u}(t)
&= 
tH_\ep(\rho)-(tH_\ep'(\rho)-u)\rho -tH_\ep(\lambda)\nn\\
&\qquad \qquad+(tH_\ep'(\rho)-u)\lambda
+c_1\nn\\[3pt]
&= -\tfrac12 tH_\ep''(\rho)(\rho-\lambda)^2+u(\rho-\lambda)+c_1 \nn\\
&=-t\ep^{1/2}(\rho-\lambda)^2+u(\rho-\lambda)+c_1.  \label{line1.8}
\end{align}

\noindent {\it Case 1.} $B\ep^{1/3}t^{2/3} \le u\le 5\rho\ep^{1/2} t$.
Note that   $20t/3>5\rho\ep^{1/2} t$.
Choose
\be
\lambda=\rho-\tfrac1{10} {u}\ep^{-1/2} t^{-1}
\quad\text{and}\quad
k=\left\lfloor\tfrac1{20} {u^2}\ep^{-1/2} t^{-1}\right\rfloor-3.
\label{eq:lakchoice}\ee
By  assuming $t\ge C(B)\ep^{-1/2}$ we guarantee that
$u\ge 1$ and 
\[ \tfrac1{40} {u^2}\ep^{-1/2} t^{-1} \ge \tfrac1{40}B^2\ep^{1/6} t^{1/3}
\ge 4. \]
Then 
\be
k\ge \tfrac1{40}{u^2}\ep^{-1/2} t^{-1}\ge 4. 
\label{eq:kprop}\ee
In the next inequality below
 the $-3$ in the definition \eqref{eq:lakchoice}
of  $k$ absorbs $c_1$ from line \eqref{line1.8}.
Let   $\overline  X=X-EX$ denote a centered random
variable.  Continuing with the   probability
 from line \eqref{line2}:
\begin{align}
&\Pv_\ep\{  J^\zeta_{\fl{V^\rho_\ep t}-u}(t) \; - \; J^\eta_{\fl{V^\rho_\ep t}-u}(t)
\le k\}\nn\\
&\le \Pv_\ep\Bigl\{  \overline  J^\zeta_{\fl{V^\rho_\ep t}-u}(t) \; - \; 
\overline  J^\eta_{\fl{V^\rho_\ep t}-u}(t) \le  
-\,\tfrac1{25}{u^2}{\ep^{-1/2}t^{-1}} \Bigr\}\nn\\
&\le C\ep t^2u^{-4} 
\Vv_\ep\Bigl[ J^\zeta_{\fl{V^\rho_\ep t}-u}(t) \; - 
\; J^\eta_{\fl{V^\rho_\ep t}-u}(t)\Bigr]
\nn\\
&\le C\ep t^2u^{-4} 
\Bigl(\Vv_\ep\bigl[ J^\zeta_{\fl{V^\rho_\ep t}-u}(t)\bigr] 
\; + \; \Vv_\ep\bigl[J^\eta_{\fl{V^\rho_\ep t}-u}(t)\bigr]\,\Bigr).\label{line6}
\end{align} 
  $C$ is a constant that can change from line to
line but  is independent of all  parameters. 

We develop bounds on  the variances above, 
first for $J^\zeta$. Utilize  the coupling with a  stationary density $\rho$ 
 process.  Then apply the basic identity 
 \be \Var^\rho_\ep\bigl[ J_x(t)\bigr]= \rho(1-\rho) 
\Ev^\rho_\ep\bigl\lvert {\sc}(t)-x\rvert 
\label{JQ}\ee
that links the variance of the current with the second class  particle. 
(This is proved in Corollary 2.4  in \cite{bala-sepp-07JSP} and  in 
Theorem 2.1 in \cite{bala-sepp-08ALEA}.)
We find 
 \begin{align}
\Vv_\ep\bigl[ J^\zeta_{\fl{V^\rho_\ep t}-u}(t)\bigr]
&\le 2\Var^\rho_\ep\bigl[ J_{\fl{V^\rho_\ep t}-u}(t)\bigr]+2 \nn\\
&=
2\rho(1-\rho) 
\Ev_\ep\bigl\lvert {\sc}(t)-\fl{V^\rho_\ep t}+u\bigr\rvert + 2\nn\\
& \le  \Ev^\rho_\ep\lvert {\sc}(t)-V^\rho_\ep t\rvert + 4u.   \label{eq:aux7}
\end{align}  
For the second variance on line \eqref{line6} we begin in the same way:
\begin{align}
\Vv_\ep\bigl[ J^\eta_{\fl{V^\rho_\ep t}-u}(t)\bigr]
&\le 2\Var^\lambda_\ep\bigl[ J_{\fl{V^\rho_\ep t}-u}(t)\bigr]+2 \nn\\
&\le 
2\lambda(1-\lambda)\Ev_\ep^\lambda\bigl\lvert {\sc}^\lambda(t)-\fl{V^\rho_\ep t}+u\bigr\rvert + 2\nn\\
&\le
\Ev_\ep^\lambda\lvert {\sc}^\lambda(t)-V^\rho_\ep t\,\rvert + 4u. \label{line6.7}\end{align}
Here we switched to a stationary density $\lambda$ process
and introduced a second class  particle ${\sc}^\lambda$ in this process.  
In order to get the same bound  as on line \eqref{eq:aux7}  we wish to
switch from ${\sc}^\lambda(t)$ to the second class particle ${\sc}(t)$ in the density-$\rho$
process.  To this end we  utilize
a coupling developed in Section 3 of \cite{bala-sepp-08ALEA}. 
Because the density $\rho$ process has higher particle density than  the 
density $\lambda$ process, the second class  particle in density $\lambda$ moves
on average faster in the direction of the drift.   
Theorem 3.1 of \cite{bala-sepp-08ALEA} allows us to couple ${\sc}^\lambda$ 
and ${\sc}$ so that  ${\sc}(t)\ge {\sc}^\lambda(t)$ with probability 1.  
 Thus continuing from line \eqref{line6.7}, \begin{align}
\Vv_\ep\bigl[ J^\eta_{\fl{V^\rho_\ep t}-u}(t)\bigr]
& \le
\Ev_\ep\bigl[{\sc}(t)-{\sc}^\lambda(t)\bigr] +
\Ev^\rho_\ep\lvert \sc(t)-V^\rho_\ep t\rvert + 4u
\label{eq:aux9}
\end{align}  
Now $\Ev_\ep\bigl[{\sc}(t)-{\sc}^\lambda(t)\bigr]= (V^\rho_\ep  -V^\lambda_\ep) t=  2\ep^{1/2} t(\rho-\lambda)$ and from the choice \eqref{eq:lakchoice} of $\lambda$, $2\ep^{1/2} t(\rho-\lambda)\le u$, hence
\begin{align}
\Vv_\ep\bigl[ J^\eta_{\fl{V^\rho_\ep t}-u}(t)\bigr]
& \le \Ev^\rho_\ep\lvert \sc(t)-V^\rho_\ep t\rvert + 5u. 
\label{eq:aux9.1}
\end{align}  
Insert bounds \eqref{eq:aux7} and \eqref{eq:aux9}
into \eqref{line6} to get
\begin{align}
\Pv_\ep ( J^\zeta_{\fl{V^\rho_\ep t}-u}(t) \; - \; J^\eta_{\fl{V^\rho_\ep t}-u}(t)
\le k)
\le C\ep 
t^2u^{-4}\Ev^\rho_\ep\lvert \sc(t)-V^\rho_\ep t\rvert+ C\eps t^2u^{-3} .
\label{eq:aux13} \end{align}
Insert \eqref{eq:kprop}  and \eqref{eq:aux13} into line \eqref{line2} 
to get 
\be 
\Pv_\ep ({\sc}(t)\le V^\rho_\ep t-u ) \le 
 C\ep t^2u^{-4}\Ev^\rho_\ep\lvert \sc(t)-V^\rho_\ep t\rvert+ C\eps t^2u^{-3} + e^{-u^2/40t} 
\label{eq:aux14}\ee
and we have verified  \eqref{eq:UBge1/2} for $\Pv^\rho_\ep({\sc}(t)\le V^\rho_\ep t - u)$
for {\it Case 1}.  

\medskip
\noindent{\it Case 2.}   $u\ge 5\rho\ep^{1/2} t$.
 Let $Z_t$ be  a nearest-neighbor random
walk   with rates $p=1/2$ to the right and 
$q=1/2+\ep^{1/2}$ to the left.   We have the stochastic domination 
$Z_t\le {\sc}(t)$ because no matter what the environment next to ${\sc}(t)$,
it has a weaker left drift than $Z_t$. 
Then, since $V^\rho_\ep=-\ep^{1/2}(1-2\rho)$,  $2\rho\ep^{1/2} t\le 2u/5$, 
and $\ep<1/4$, 
\be
\Pv_\ep^\rho\{{\sc}(t)\le V^\rho_\ep t- u\} 
\le P\{Z_t\le -\ep^{1/2} t- \tfrac35u\}. 
\label{eq:aux15.5a}
\ee
For $\alpha\in(0,1]$, utilizing 
$({e^\alpha+e^{-\alpha}})/2\le 1+\alpha^2$ and $e^{-\alpha}\ge 1-\alpha$,
\begin{align*}
\Ev_\ep[e^{-\alpha Z_t}]
&=\exp\Bigl( -(1+\ep^{1/2})t+t\frac{e^\alpha+e^{-\alpha}}2 (1+2\ep^{1/2})
-\ep^{1/2} t e^{-\alpha}\,\Bigr)\\
&\le \exp\bigl( \alpha^2t + (\alpha+2\alpha^2) \ep^{1/2} t \bigr).   
\end{align*} 
We can estimate $P\{Z_t\le -\ep^{1/2} t- \tfrac35u\}\le  \exp\bigl(-\tfrac35\alpha u+ 2\alpha^2t  \bigr)$ and choose $\alpha =1\wedge\frac{3u}{20t}$ to obtain
\be\begin{split}
\Pv_\ep^\rho\{{\sc}(t)\le V^\rho_\ep t- u\} 
& \le\begin{cases}  \exp(-\tfrac{9}{200}u^2t^{-1}) &u\le 20t/3\\
              \exp(-3u/10)  &u> 20t/3. \end{cases} 
\end{split} 
\label{eq:aux15.5}\ee
Combining   \eqref{eq:aux14}  and  
 \eqref{eq:aux15.5}  
gives Lemma \ref{lm:UB1} for $\Pv^\rho_\ep({\sc}(t)\le V^\rho_\ep t - u)$. 



%
  The corresponding upper tail bound $\Pv^\rho_\ep({\sc}(t)- V^\rho_\ep t \ge  u)$ is obtained
from that for  $\Pv^\rho_\ep({\sc}(t)-V^\rho_\ep t \le - u)$ by a particle-hole interchange followed by
a reflection of the lattice.  For details we refer to Lemma 5.3
in \cite{bala-sepp-08ALEA}.   This completes the proof of Lemma \ref{lm:UB1}.
\end{proof}

 \begin{proof}[Proof of the upper bound of Theorem \ref{Qmomthm}]
Integrate Lemma \ref{lm:UB1} to get the  bound \eqref{Qmom}  on the moments
of the second class particle.  First for $m=1$.  
\begin{align*}
& \Ev^\rho_\ep\lvert  {\sc}(t)-V^\rho_\ep t\rvert= \int_0^\infty 
\Pv_\ep^\rho\{\,\lvert {\sc}(t) -V^\rho_\ep t\rvert \ge u\}\,du\\
&\quad \le \tfrac13{C}{B^{-3}}\Ev^\rho_\ep\lvert  {\sc}(t)-V^\rho_\ep t\rvert + 
\Bigl(B+\tfrac12{C}{B^{-2} } \Bigr) \ep^{1/3}t^{2/3}  
\\&\quad\quad+C_1(B)t^{1/3}\ep^{-1/3} \exp\{-\tfrac1{C_1(B)} \ep^{2/3}t^{1/3}\}  
+ \tfrac{20}3 e^{-t/C}.
\end{align*}
$C_1(B)$ is a new constant that depends on $B$.
Set $B=C^{1/3}$  to obtain
\begin{align*}
\Ev^\rho_\ep\lvert \sc (t)-V^\rho_\ep t\rvert&\le \tfrac94 {C^{1/3}} \ep^{1/3}t^{2/3}  
+ C_1t^{1/3}\ep^{-1/3}\exp\{-\tfrac1{C_1} t^{1/3}\ep^{2/3} \}
+  \tfrac{20}3 e^{-t/C}.
\end{align*}
We can fix a constant $c_0$ large enough so that,
 for a new constant $C$, 
\be
\Ev^\rho_\ep\lvert  {\sc}(t)-V^\rho_\ep t\rvert\le C \ep^{1/3}t^{2/3}   
\quad\text{provided $t\ge c_0\ep^{-2} $.}  \label{eq:Psibd2}\ee
Restrict to $t$ that satisfy this requirement and  substitute \eqref{eq:Psibd2}
 into Lemma \ref{lm:UB1}. 
Then upon using $u\ge B\ep^{1/3}t^{2/3} $  and redefining
$C$ once more, we have 
   for $B\ep^{1/3}t^{2/3} \le  u\le 20t/3$:
\be \begin{aligned} 
\Pv_\ep^\rho (\,\lvert {\sc}(t) -V^\rho_\ep t\rvert \ge u) \le 
 C\ep t^2 u^{-3} + 2e^{-u^2/Ct}. 
 \end{aligned}\label{eq:aux18} \ee
  Now take $1<m<3$ 
and use   \eqref{eq:aux18} together with Lemma \ref{lm:UB1}
\begin{align}
&\Ev_\ep^\rho\lvert {\sc}(t)-V^\rho_\ep t\rvert^m 
=m \int_0^\infty \Pv_\ep^\rho\{\,\lvert {\sc}(t) -V^\rho_\ep t\rvert \ge u\}
u^{m-1}\,du \le 
B^m \ep^{m/3} t^{2m/3} \nn\\&  \;+\; Cm\ep t^2 \int_{B\ep^{1/3}t^{2/3} }^\infty
({u^{m-4}}+2m e^{-u^2/Ct} u^{m-1})\,du
\;+\;2m\int_{20t/3}^\infty e^{-u/C} u^{m-1}\,du.\nn
\end{align}
This  gives  
$
\Ev_\ep^\rho\lvert {\sc}(t)-V^\rho_\ep t\rvert^m  \le 
 \frac{C}{3-m} \ep^{m/3} t^{2m/3}   $
 provided $t\ge c_0\ep^{-2}$ for a large enough  $c_0$.    
 \end{proof}  
 
 \subsection{Proof of the lower bound of Theorem \ref{Qmomthm}}

By Jensen's inequality it suffices to prove the lower bound for $m=1$. 
 Let  $C_{UB}$ denote the constant in the upper bound statement that we just proved.  
 We can also assume $c_0\ge 1$.  
 Fix a constant $b>0$ and  set  
\begin{align*}
a_1=2C_{UB}+1  
\quad\text{and}\quad 
a_2=8 + \sqrt{32 b}+8 \sqrt{C_{UB}}.  \end{align*}
Increase  $b$ if necessary so that 
\be  b^2-2a_2\ge 1.  \label{ba_2}\ee

Fix a density $\rho\in(0,1)$ and define an auxiliary density 
$\lambda=\rho-bt^{-1/3}\ep^{-1/6}$. 
 Define  positive
integers 
\be
u=\fl{a_1t^{2/3}\ep^{1/3}} \quad\text{and}\quad 
n=\fl{V^\rho t}-\fl{V^\lambda t}+u.
\label{def:n}\ee
By taking $c_0$ large enough in the statement of Theorem \ref{Qmomthm} 
we can ensure that $\lambda\in(\rho/2,\rho)$ and $u\in\N$.

Construct a basic coupling of
three processes $\eta\le\eta^+\le \zeta$ with the following initial state:  

\medskip

(a) Initially $\eta$ has
i.i.d.~Bernoulli($\lambda$) occupations 
$\{\eta(0,x):x\ne n\}$ and $\eta(0,n)=0$. 

(b)  Initially $\eta^+(0,x)=\eta(0,x)+\delta_{n}(x)$ for all $x\in\Z$. 
${\sc}^{(n)}(t)$ is the location of the unique
discrepancy between $\eta(t)$ and $\eta^+(t)$. 

(c)  Initially $\zeta$ has independent 
occupation variables, coupled with $\eta(0)$
as follows: 

(c.1) $\zeta(0,x)=\eta(0,x)$ for $0\le x<n$ and 
  $\zeta(0,n)=1$.

 (c.2)   For   $x>n$ and $x<0$ variables $\zeta(0,x)$ 
are i.i.d.~Bernoulli($\rho$) and 
$\zeta(0,x)\ge\eta(0,x)$. 

Let $\Pv$ denote  the probability measure of the coupled processes. 
Label the $\zeta-\eta$ second class particles
as $\{X_m(t):m\in\Z\}$ 
so that initially 
\[
\dotsm< X_{-1}(0)<0<X_{0}(0)= n={\sc}^{(n)}(0)< X_{1}(0)<X_{2}(0)<\dotsm
\]
Let again the random label $m(t)$ satisfy 
 ${\sc}^{(n)}(t)=X_{m(t)}(t)$, with initial value  
  $m(0)=0$. 
In basic coupling
 $m(\cdot)$ jumps to the left
with rate $q$ and to the right with rate $p$, but only
when there is an $X$ particle  adjacent to $X_{m(\cdot)}$.
As in the proof of the upper bound,  
Lemma 5.2 in \cite{bala-sepp-08ALEA} gives the bound 
\be
\Pv\{m(t)\ge k\} \le  \exp\{-\ep^{1/2}k\}
 \quad\text{for all $t\ge 0$ and $k\ge 0$.}
\label{mQ2}\ee
By the upper bound already proved and by the choice of $a_1$,  
\be\begin{aligned}
& \Pv\{ {\sc}^{(n)}(t)\le \fl{V^\rho t}\}
=\Pv\{ {\sc}^{(n)}(t)\le n+\fl{V^\lambda t} -u\} \\
&\qquad \le  u^{-1}  \Ev\abs{ {\sc}^{(n)}(t)- n-\fl{V^\lambda t} } 
\le \frac{C_{UB}t^{2/3}\ep^{1/3}}{\fl{a_1t^{2/3}\ep^{1/3}}}\le \tfrac12. 
\end{aligned}\label{lb:ucond1}\ee
This gives a lower bound for the complementary event,
\be\begin{split}
\tfrac12 &\le \Pv\{ {\sc}^{(n)}(t)> \fl{V^\rho t}\}
\le \Pv\{m(t)\ge k\}
+\Pv\{ J^\zeta_{\fl{V^\rho t}}(t)-J^\eta_{\fl{V^\rho t}}(t) \ge -k\}. 
\end{split}\label{eq:aux36}\ee
The reasoning behind the second inequality above is as follows:  
${\sc}^{(n)}(t)> \fl{V^\rho t}$ and  $m(t)<k$ imply 
 $X_{k}(t)>  \fl{V^\rho t}$   and consequently   
$
-k\le  J^{\zeta-\eta}_{\fl{V^\rho t}}(t) = J^\zeta_{\fl{V^\rho t}}(t)-J^\eta_{\fl{V^\rho t}}(t)$.
 
 Put $k=\fl{a_2t^{1/3}\ep^{1/6}}-2$.  
 Observe from \eqref{mQ2} that $ \Pv\{m(t)\ge k\}\le e^{-2}<1/4 $
  follows from $a_2t^{1/3}\ep^{1/6} \ge 2\ep^{-1/2} +3$, which is guaranteed by
  $t\ge c_0\ep^{-2}$ and the definition of $a_2$. Hence
\begin{align}
\tfrac14&\le 
\Pv\{ J^\zeta_{\fl{V^\rho t}}(t)-J^\eta_{\fl{V^\rho t}}(t) \ge
-a_2t^{1/3}\ep^{1/6}+2\} \nn\\
&\quad \le 
\Pv\{ J^\zeta_{\fl{V^\rho t}}(t) \ge -2a_2t^{1/3}\ep^{1/6} 
-t\ep^{1/2}(2\rho\lambda-\lambda^2) +1 \} 
\nn\\
&\quad\quad + 
\Pv\{ J^\eta_{\fl{V^\rho t}}(t) \le  -a_2t^{1/3}\ep^{1/6} 
-t\ep^{1/2}(2\rho\lambda-\lambda^2)-1\}. 
\label{line19}
\end{align}
Consider  line \eqref{line19}.   The 
$\eta$ process can be coupled with a stationary $P^\lambda$-process
with at most one discrepancy. The mean current in
the stationary process is  
\begin{align*}
E^\lambda[ J_{\fl{V^\rho t}}(t)]
&= tH(\lambda)-\lambda\fl{V^\rho t}
\ge tH(\lambda)-\lambda{V^\rho t}
=-t\ep^{1/2}(2\rho\lambda-\lambda^2). 
\end{align*}
Hence 
\begin{align}
&\text{line \eqref{line19}} 
\le 
P^\lambda\{ J_{\fl{V^\rho t}}(t) \le  -a_2t^{1/3}\ep^{1/6} 
-t\ep^{1/2}(2\rho\lambda-\lambda^2)\}\nn\\[3pt]
&\le P^\lambda\bigl\{ \overline J_{\fl{V^\rho t}}(t) \le  -a_2t^{1/3}\ep^{1/6} \bigr\}
\le a_2^{-2}t^{-2/3}\ep^{-1/3}
\Var^\lambda\bigl[ J_{\fl{V^\rho t}}(t)\bigr]\nn\\[3pt]
&\le \frac{\Ev^\lambda\lvert {\sc}(t)-\fl{V^\rho t}\rvert}{a_2^2t^{2/3}\ep^{1/3}}
\; \le \; \frac{\Ev^\lambda\lvert {\sc}(t)-V^\lambda t\rvert}{a_2^2t^{2/3}\ep^{1/3}}
+ \frac{2 b }{a_2^2 }
+ \frac{ 1}{a_2^2t^{2/3}\ep^{1/3}}\nn\\
&\le C_{UB}a_2^{-2} + \tfrac1{16} + \tfrac1{64}\le \tfrac18.
\label{line22}
\end{align} 
After Chebyshev above  we applied the basic identity 
\eqref{JQ}  for which we introduced a second class particle 
${\sc}(t)$ in a density $\lambda$ system under the measure
 $\Pv^\lambda$. 
Then 
we replaced $\fl{V^\rho t}$ with $V^\lambda t$ and    applied  the upper
bound    and 
properties of $a_2$.  

Put this last bound back into line \eqref{line19} to get 
\be
\tfrac18\le 
\Pv(\mathcal A) \quad{\rm where} \quad \mathcal A=\{ J_{\fl{V^\rho t}}(t) \ge -2a_2t^{1/3}\ep^{1/6} 
-t\ep^{1/2}(2\rho\lambda-\lambda^2) +1 \}
\label{line24}\ee
Let $\gamma$
denote the distribution of the initial $\zeta(0)$ 
configuration described by (a)--(c) in the beginning
of this section. As before $\nu^\rho$ is the
density $\rho$ i.i.d.~Bernoulli measure.  The Radon-Nikodym
derivative is 
\[
\frac{d\gamma}{d\nu^\rho}(\omega)= 
\frac1{\rho}\ind\{\omega_{-n}=1\}\cdot
\prod_{i=-n+1}^0\Bigl( \frac{\lambda}{\rho}\ind\{\omega_i=1\}
+ \frac{1-\lambda}{1-\rho}\ind\{\omega_i=0\}\Bigr).
\]
Bound its second moment: 
\be\begin{aligned}
E^\rho[|\frac{d\gamma}{d\nu^\rho}|^2]&=\frac1\rho\Bigl(1+\frac{(\rho-\lambda)^2}{\rho(1-\rho)}
\Bigr)^n 
\le \rho^{-1}e^{n(\rho-\lambda)^2/\rho(1-\rho)}
\le  c_2(\rho) .
\end{aligned} \label{def:c1}\ee
Here condition $t\ge c_0\ep^{-2}$ implies a bound  $c_2(\rho)<\infty$  
 independent of $t$ and $\ep$.
 From \eqref{line24} and Schwarz's inequality 
\begin{align}
\tfrac18\le \Pv(\mathcal A)
&= \int P^\omega(\mathcal A)\,\gamma(d\omega) \nn\\ 
&= \int P^\omega(\mathcal A)\frac{d\gamma}{d\nu^\rho}(\omega)\,\nu^\rho(d\omega)
\le c_2(\rho)^{1/2} \bigl(P^\rho(\mathcal A)\bigr)^{1/2}.
\label{line26} 
\end{align}
 Note the stationary mean  
\begin{align*}
E^\rho\bigl[J_{\fl{V^\rho t}}(t)\bigr]
=-t\ep^{1/2}\rho^2+\rho V^\rho t-\rho\fl{V^\rho t}
\le - t\ep^{1/2}\rho^2+1.
\end{align*}
 Continue from line \eqref{line26}, recalling    
\eqref{ba_2}:   
\begin{align*}
(64c_2(\rho))^{-1}&\le P^\rho(\mathcal A)= P^\rho\{ J_{\fl{V^\rho t}}(t) \ge -2a_2t^{1/3}\ep^{1/6} 
-t\ep^{1/2}(2\rho\lambda-\lambda^2)+1 \}\\
&\le P^\rho\{ \,\overline J_{\fl{V^\rho t}}(t) \ge -2a_2t^{1/3}\ep^{1/6} 
+t\ep^{1/2}(\rho-\lambda)^2 \}\\
&= P^\rho\{ \,\overline J_{\fl{V^\rho t}}(t) 
\ge (  b^2-2a_2)t^{1/3}\ep^{1/6} \}   
\le P^\rho\{ \, \overline J_{\fl{V^\rho t}}(t) 
\ge t^{1/3}\ep^{1/6} \}\\
&\le  t^{-2/3}\ep^{-1/3}
\Var^\rho \bigl[J_{\fl{V^\rho t}}(t)\bigr] \; \le \; t^{-2/3} \ep^{-1/3}  \Ev^\rho\abs{{\sc}(t)-V^\rho t}. 
\end{align*}
This completes the proof of  the lower bound and thereby the proof of Theorem 
\ref{Qmomthm}.

\section{Appendix:  Properties of the solution}
\label{properties}
\setcounter{equation}{0}

Throughout this section $h(t,x)=-\log Z(t,x)$, where $Z(t,x)$ is the 
solution of (\ref{heat}) starting with a two sided Brownian motion $\{B(x):x\in\R\}$ with $B(0)=0$.  The goal is to
prove
\be \Var (h(t,x))-|x| \to 0 \quad\text{ as $\abs{x}\to\infty$.}
\label{auxh0a}\ee

Define the current of $u=\partial_x h$ across $x$ up to time $t$ by
\[
N(t,x) = h(t,x)- h(0,x)
\]
and the mass of $u$ in the interval $[0,x]$ at time $t$ by
\[
M(t,x) = h(t,x) - h(t,0).
\]
\begin{prop}
$
{\rm Var}( h(t,x) ) -|x| = \Cov  (N(t,0), N(t,x)) 
$.
\end{prop}

\begin{proof}   Since $h(0,0)=0$ we have $h(t,x) = M(t,x) +N(t,0)$.  From the invariance
of white noise ($5$ of Prop \ref{thm1}),  ${\rm Var}(M(t,x))=|x|$.  Hence 
\begin{equation}\label{hg}
{\rm Var}(h(t,x))- |x| = {\rm Var}(N(t,0)) +2 \Cov (M(t,x), N(t,0)).
\end{equation}
We claim that
\be\begin{aligned}
\Cov (M(t,x), N(t,0)) &=  - {\rm Var}(N(t,0)) 
+ \Cov (M(t,x), N(t,x)) \\
&\qquad +\Cov (N(t,x), N(t,0)).
\end{aligned}\label{hh}\ee
To see this, note that we always have the conservation law \[N(t,x)-N(t,0)= M(t,x)-M(0,x).\]  Hence
\begin{align*}  \Cov (M(t,x), N(t,0)) &=   
\Cov (M(t,x), N(t,x)) \\
&\qquad -  \Cov (M(t,x), (M(t,x)- M(0,x))).\end{align*}   But 
 $\Cov (M(t,x), (M(t,x)- M(0,x))) = \frac12 {\rm Var}(M(t,x)- M(0,x))$.  By the conservation law again
$\frac12 {\rm Var}(M(t,x)- M(0,x))=  \frac12 {\rm Var}(N(t,x)-N(t,0))$.  Finally, by the translation invariance,
$\frac12 {\rm Var}(N(t,x)-N(t,0))=  {\rm Var}(N(t,0))-\Cov (N(t,x), N(t,0))$.  This gives (\ref{hh}).

 From (\ref{hh}) we can rewrite the right hand side of (\ref{hg}) as
\begin{equation}\Cov (N(t,x), N(t,0))+ \Cov (M(t,x), N(t,0)+ N(t,x)).
\end{equation} 
The proof is completed by noting that the second term vanishes by symmetry.  To see it, note that 
${\rm Var}(h(t,-x))= {\rm Var}(h(t,x))$ and by translation invariance $\Cov (N(t,-x), N(t,0))= \Cov (N(t,x), N(t,0))$.  Hence  $\Cov (M(t,x), N(t,0)+ N(t,x))=  \Cov (M(t,-x), N(t,0)+ N(t,-x))$.  But translating by $x$ gives
$ \Cov (M(t,-x), N(t,0)+ N(t,-x))= 
 \Cov (-M(t,x), N(t,x)+ N(t,0))$.
\end{proof}

 \begin{prop} \label{14h} $\lim_{|x|\to \infty} \Cov  ( N(t,0), N(t,x) ) = 0.$
\end{prop}

The two propositions combine to prove 
\eqref{auxh0a}.


\medskip

The proof of Proposition \ref{14h} is based on the following lemma.  We need some notation.
Fix $R>0$ and let $ W_1(t,x),  W_2(t,x)$ be cylindrical Wiener processes and $B_1(x), B_2(x)$  two sided Brownian motions with $B_1(0)=B_2(0)=0$, coupled as follows:   For any $\varphi \in C_c^\infty(\mathbb R)$ supported in 
$(-\infty,R)$, $\langle \varphi, W_1(t)\rangle = \langle \varphi, W_2(t)\rangle$ are independent of $\int \varphi dB_1= \int\varphi dB_2$, while for any  $\varphi \in C_c^\infty(\mathbb R)$ supported in 
$(R,\infty)$, $\langle \varphi, W_1(t)\rangle$, $ \langle \varphi, W_2(t)\rangle$, $\int \varphi dB_1$ and $ \int\varphi dB_2$ are independent. We will say that $(W_1,dB_1)$ and $(W_2,dB_2)$ are the same on $(-\infty,R)$ and independent on $(R,\infty)$.  
\begin{lemma}\label{lemmaone} 
Let $Z_i(t,x)$, $i=1,2$,  be the solutions of (\ref{heat}) with $W_i$, $i=1,2$ and 
initial data $Z_i(0,x) = \exp\{ B_i(x) \}$, where $(W_1,dB_1)$ and $(W_2,dB_2)$ are the same on $(-\infty,R)$ and independent on $(R,\infty)$.
 Then there is a finite $C$ such that for $R\ge |x|+ 2t$, \begin{equation}
E[ ( Z_1(t,x)- Z_2(t,x))^2 ]\le Ce^{-R + C( t+|x|)} .
\end{equation}
\end{lemma}

\begin{proof}   Let $p(t,x) = \frac1{\sqrt{2\pi t}} e^{-x^2/2t}$ be the heat kernel.  We can write
\begin{equation}\label{for}
Z_i(t,x) =\!\!\! \int \! p(t,x-y) e^{ B_i(y)} dy - \int \int_0^t p(t-s,x-y) Z_i(s,y)W_i(dsdy).
\end{equation}
First we obtain a preliminary bound
on $E[ Z^2_i(t,x)]$.  By Schwarz's inequality  it is bounded above by  (dropping the $i$ for clarity),
\begin{equation}2 E[ (\int\!\! p(t,x-y) e^{ B(y)} dy)^2] + 2E[ (  \int \!\!\int_0^t p(t-s,x-y) Z(s,y)W(dsdy))^2].\nonumber
\end{equation}
By Jensen's inequality,    \[ E[ (\int p(t,x-y) \exp\{ B(y)\} dy)^2]\le \int p(t,x-y) e^{2|y|} dy.
\] 
 Furthermore,
\begin{align*}
&E[ (  \int\int_0^t  p({\scriptstyle{t-s,x-y}}) Z(s,y)W(dsdy))^2]\\
&\qquad=\int \int_0^t p^2({\scriptstyle{t-s,x-y}}) E[ Z^2(s,y)] dsdy\nonumber \\
&\qquad\le \int \int_0^t {\scriptstyle\frac1{\sqrt{ \pi (t-s)}}} p({\scriptstyle{t-s,x-y}}) E[ Z^2(s,y)] dsdy. 
\end{align*}
Call $g(t,x)=E[ Z^2_i(t,x)]$ and let $P_{t}$ denote the heat semigroup.  $g(0,x)= e^{2|x|}$ and have shown that
\begin{equation}
g(t) \le P_t g(0) + \int_0^t {\scriptstyle\frac1{\sqrt{ \pi (t-s)}}} P_{t-s} g(s) ds.
\end{equation}
Iterating once we see that this is bounded above by
\begin{equation}
(1+ s\sqrt{ t/\pi })  P_tg(0) +   \int_0^t {\scriptstyle\frac1{\sqrt{ \pi (t-s)}}} P_{t-s}  \int_0^s {\scriptstyle\frac1{\sqrt{ \pi (s-u)}}} P_{s-u} g(u) du ds.\end{equation}
The last term can be simplified by noting that $P_{t-s}P_{s-u}=P_{s-u}$, applying Fubini's theorem, and using
$\int_0^s {\scriptstyle\frac{du}{\sqrt{(t-s)(s-u)}}} = \pi$.  The result is
\begin{equation}\label{in6}
g(t)\le (1+ s\sqrt{ t/\pi }) P_tg(0) + \int_0^t  P_{t-s}   g(s)  ds.
\end{equation}
If we let $\bar g(0)=g(0)$ and $\bar g(t)$ satisfy (\ref{in6}) with 
equality instead of inequality, then $\bar g- g$ satisfies $(\bar g- g)(t)\ge  \int_0^t  P_{t-s}   (\bar g- g)(s)  ds$ with $(\bar g- g)(0)=0$.
By the maximum principle for the heat equation, $g\le \bar g$.  $\bar g(t,x)$ is readily computed with the result that for some finite $C$,
\begin{equation}\label{prelim}
E[ Z^2_i(t,x)]=g(t,x)\le C e^{ C( t + |x|) }.
\end{equation}

By (\ref{for}) again  we have 
$
f(t,x): = E[ ( Z_1(t,x)-Z_2(t,x))^2] 
$ is bounded above by twice
\begin{eqnarray}
&\hskip-1.1in E\Big[ \Big( \int p(t,x-y) \big(\exp\{ B_1(y)\}  \exp\{ B_2(y)\}\big) dy\Big)^2 \Big]&\label{p1}\\\label{p2}
&\hskip-.2in+E\Big[ \Big(\int\!\! \int_0^t p({\scriptstyle{t-s,x-y}}) \big(Z_1({\scriptstyle{s,y}})W_1(dsdy) - Z_2({\scriptstyle{s,y}})W_2(dsdy)\big)\Big)^2 \Big].&
\end{eqnarray}Explicit computation gives that (\ref{p1})  is equal to
\begin{equation}
 \int_R^\infty \!\!\! \int_R^{\infty} \!\!\! ~p(t,x-y_1)p(t,x-y_2) E[ (e^{B_1(y_1)}-e^{B_2(y_1)}) (e^{B_1(y_2)}-e^{B_2(y_2)}) ]dy_1dy_2\nonumber\end{equation} By Schwarz's inequality, 
$
(\ref{p1}) \le 2  \int p(t,x-y)1_{\{y\ge R\}} e^{2y} dy $.
Another explicit computation gives that (\ref{p2}) is equal to \begin{eqnarray}
  & & \hskip-.4in\int_0^t \int_{-\infty}^R p^2({\scriptstyle{t-s,x-y}}) f(s,y) dy ds+
2\int_0^t \int_R^\infty p^2({\scriptstyle{t-s,x-y}}) E[Z^2_1(s,y)] dy ds\nonumber\\
& \!\!\!\le \!\!\! & \int_0^t {\scriptstyle\frac1{\sqrt{\pi(t-s)}}}\int p({\scriptstyle{t-s,x-y}}) (f(s,y) +
1_{\{y\ge R\}} Ce^{C( s+|y|)}) dy ds.\nonumber
\end{eqnarray}
%
Hence $f(t)$ satisfies the same equation as $g(t)$ in (\ref{in6})
except that this time $f(0,x) \le 1_{\{y\ge R\}} e^{2y}$.
The same argument now shows that
there is a finite $C$ such that  $f(t,x) \le Ce^{-R + C( t+|x|)}$ for $R\ge |x|+ 2t$.
\end{proof}

\begin{proof}[Proof of Proposition \ref{14h}] Let us use the notation $\bar N(t,x)$ for the normalized current $N(t,x) - E[N(t,x)]$.  First of all note that by (\ref{bdsonh}) and $5$ of Proposition \ref{thm1}, $E[\bar N^2(t,x)]\le C(t)$ and does not depend on $x$.  
Now let $(W_1,dB_1)$, $(W_2,dB_2)$, $(W_3,dB_3)$ be coupled so that $(W_1,dB_1)$ and $(W_2,dB_2)$ are the same on $(-\infty, x/2)$ and independent on $(x/2,\infty)$,  $(W_2,dB_2)$ and $(W_3,dB_3)$ are the same on $( x/2,\infty)$ and independent on $(-\infty, x/2)$, and $(W_1,dB_1)$ and $(W_3,dB_3)$ are  independent.  Let $\bar N_1$, $\bar N_2$, $\bar N_3$ be the currents 
corresponding to the three different pairs.  Of course $\Cov  ( N(t,0), N(t,x) ) =E[ \bar N_1(t,0) \bar N_1(t,x)]$.  By Schwarz's inequality 
\begin{equation}
|E[ \bar N_1(t,0) \bar N_1(t,x)] - E[ \bar N_2(t,0) \bar N_1(t,x)]|\le C |E[( \bar N_1(t,0)-\bar N_2(t,0))^2]|^{1/2}
\end{equation}
and 
\begin{equation}
|E[ \bar N_1(t,0) \bar N_2(t,x)] - E[ \bar N_1(t,0) \bar N_3(t,x)]|\le C |E[( \bar N_2(t,x)-\bar N_3(t,x))^2]|^{1/2}.
\end{equation}
By independence, $E[ \bar N_1(t,0) \bar N_3(t,x)]=0$.  By symmetry, 
$$E[ \bar N_2(t,0) \bar N_1(t,x)]=E[ \bar N_1(t,0) \bar N_2(t,x)] .$$  
Hence 
\begin{equation}
\Cov  ( N(t,0), N(t,x) )  \le C (E[(  h_1(t,0)- h_2(t,0))^2])^{1/2}.
\end{equation}

For each $L>0$, let $\log_Lz = \log z$ for $z\ge L^{-1}$ and $\log_L z = -\log L$ for $0< z <L^{-1}$.
We have
\begin{eqnarray}
E[(  h_1(t,0)- h_2(t,0))^2] & \le &2 E[(  \log_L Z_1(t,0)-  \log_L Z_2(t,0))^2]\nonumber\\ && + 4E[ 1_{\{0<Z_1<L^{-1}\}} (\log Z_1(t,0))^2].
\end{eqnarray} 
Because $\log_L$ is Lipschitz with constant $L$ we have
\begin{equation}
E[ (  \log_L Z_1(t,0)-  \log_L Z_2(t,0))^2]\le L^2E[ (  Z_1(t,0)- Z_2(t,0))^2].
\end{equation}
By Lemma \ref{lemmaone}, for each fixed $L$, this vanishes as $|x|\to \infty$.  On the other hand, since $E[ (\log Z_1(t,0))^2]<\infty$, by the dominated convergence theorem, 
\begin{equation}
\lim_{L\to \infty} E[ 1_{\{0<Z_1<L^{-1}\}} (\log Z_1(t,0))^2]=0.
\end{equation}
This completes the proof.
\end{proof}

\bibliographystyle{plain}



\bibliography{bqsscalingrefs}

\end{document}